\documentclass[reqno, a4paper,12pt
]{amsart}

\usepackage{ amssymb, amsmath, enumerate, amsfonts, amsthm, mathrsfs, url, bm, mathtools}

\usepackage{tkz-euclide}

\usepackage{xcolor}  	
\usepackage[backref]{hyperref}
\hypersetup{
	colorlinks,
    linkcolor={blue!60!black},
    citecolor={blue!60!black},
    urlcolor={red!60!black}
}

\usepackage{color}

\numberwithin{equation}{section}

 \newcommand{\comment}[1]{}  

\usepackage{ amssymb, amsmath, enumerate, amsfonts, amsthm, mathrsfs, url, bm}

\numberwithin{equation}{section}

\usepackage{xcolor}  	
\usepackage[backref]{hyperref}
\hypersetup{
	colorlinks,
    linkcolor={blue!60!black},
    citecolor={blue!60!black},
    urlcolor={red!60!black}
}

\usepackage{color}

\usepackage[margin=1.5in]{geometry}

\RequirePackage{doi}

\usepackage[square,sort,comma,numbers]{natbib}
\newcommand{\C}{\mathbb{C}}

\newcommand{\N}{\mathbb{N}} 
\newcommand{\Z}{\mathbb{Z}}

\newcommand{\R}{\mathbb{R}}

\newcommand{\T}{\mathbb{T}}

\theoremstyle{claim}

\newcommand{\set}[1]{\mathopen{}\left\{#1\mathclose{}\right\}}
\newcommand{\bigset}[1]{\bigl\{ #1 \bigr\}}

\newcommand{\abs}[1]{\mathopen{}\left| #1\mathclose{}\right|}

\newcommand{\biggabs}[1]{\biggl| #1 \biggr|}

\newcommand{\sqbrac}[1]{\mathopen{}\left[ #1 \mathclose{}\right]}
\newcommand{\bigsqbrac}[1]{\bigl[ #1 \bigr]}

\newcommand{\ceil}[1]{\mathopen{}\left\lceil #1 \mathclose{}\right\rceil}

\newcommand{\floor}[1]{\mathopen{}\left\lfloor #1 \right\rfloor}
\newcommand{\brac}[1]{\mathopen{}\left( #1 \mathclose{}\right)}
\newcommand{\bigbrac}[1]{\bigl( #1 \bigr)}

\newcommand{\biggbrac}[1]{\biggl( #1 \biggr)}

\newcommand{\norm}[1]{\mathopen{}\left\| #1\mathclose{}\right\|}
\newcommand{\bignorm}[1]{\big\| #1 \big\|}

\newcommand{\recip}[1]{\frac{1}{#1}}
\newcommand{\trecip}[1]{\tfrac{1}{#1}}

\newcommand{\E}{\mathbb{E}}

\newcommand{\meas}{\mathrm{meas}}

\newcommand{\intd}{\mathrm{d}}
\newcommand{\supp}{\mathrm{supp}}

\newcommand{\eps}{\varepsilon}

\makeatletter
\let\@@pmod\pmod
\DeclareRobustCommand{\pmod}{\@ifstar\@pmods\@@pmod}
\def\@pmods#1{\mkern4mu({\operator@font mod}\mkern 6mu#1)}
\makeatother

\newtheorem{theorem}{Theorem}[section]

\newtheorem{corollary}[theorem]{Corollary}
\newtheorem{proposition}[theorem]{Proposition}
\newtheorem{lemma}[theorem]{Lemma}

\theoremstyle{definition}
\newtheorem{definition}[theorem]{Definition}
\newtheorem{remark}[theorem]{Remark}

\numberwithin{theorem}{section}

\renewcommand{\leq}{\leqslant}
\renewcommand{\geq}{\geqslant}

\begin{document}

\title[Extremal Sidon Sets]
{Extremal Sidon Sets are Fourier Uniform, with Applications to Partition Regularity}

\author{Miquel Ortega}
\address{Universitat Politècnica de Catalunya}
\email{miquel.ortega.sanchez-colomer@estudiantat.upc.edu}

\author{Sean Prendiville}
\address{Department of Mathematics and Statistics\\
Lancaster University
}
\email{s.prendiville@lancaster.ac.uk}



\date{\today}

\begin{abstract}
Generalising results of Erd\H{o}s-Freud and Lindstr\"om, we prove that the largest Sidon subset of a bounded interval of integers is equidistributed in Bohr neighbourhoods. We establish this by showing that extremal Sidon sets are Fourier-pseudorandom, in that they have no large non-trivial Fourier coefficients. As a further application we deduce that, for any partition regular equation in five or more variables, every finite colouring of an extremal Sidon set has a monochromatic solution.
\end{abstract}

\maketitle

\setcounter{tocdepth}{1}
\tableofcontents

\section{Introduction}

A subset $S$ of an additively-written abelian group is \emph{Sidon} if every non-zero $x$ has at most one representation as a difference $x = s_1-s_2$ with $s_1, s_2 \in S$. There have been a number of works investigating the size of Sidon sets $S \subset \Z$. Erd\H{os} and Tur\'an \cite{ErdosTuranProblem} established the well-known bound\footnote{For a proof see Appendix \ref{sec:SidonSize}.}
\begin{equation}\label{eq:SidonSize}
|(n, n+N] \cap S| \leq N^{1/2} + O(N^{1/4}).
\end{equation}
A corresponding lower bound was found by Singer \cite{SingerTheorem}, who constructed a Sidon set $S \subset[N] := \set{1,2, \dots, N}$ of size
\begin{equation}\label{eq:Singer}
|S| \geq N^{1/2} - O(N^{\alpha/2}),
\end{equation}
where $\alpha$ is a real number for which there is always a prime in $[x-x^\alpha, x]$ when $x$ is large (the current record \cite{BHPDifference} is $\alpha = 0.525$). 

Informally, we call a Sidon set $S \subset [N]$ \emph{extremal} if its size is `close' to $N^{1/2}$ in some sense. There has been speculation on the (im)possibility of characterising such sets \cite{OBryantComplete, GowersDenseSidon, forey2021algebraic, eberhard2021apparent, eberhardmanners2021apparent}. We contribute to this discussion by showing that extremal Sidon sets are Fourier pseudorandom, by which we mean that (after appropriate renormalisation) their Fourier transform behaves essentially like the Fourier transform of the ambient interval.
\begin{definition}[Fourier transform]\label{fourier transform}
For $f : \Z \to \C$ with finite support define $\hat{f} : \T \to \C$ by
$$
\hat{f}(\alpha) := \sum_{n\in \Z} f(n) e(\alpha  n).
$$
Here $e(\beta)$ stands for $e^{2\pi i \beta}$.
\end{definition}
\begin{theorem}[Fourier uniformity]\label{thm:sidon_pseudorandom}
Let $S \subset [N]$ be a Sidon set. Then
    \begin{equation}\label{eq:fourier-unif}
        \left\|\hat{1}_{S} - \tfrac{|S|}{N}\hat{1}_{[N]}\right\|_{\infty}
        \ll N^{1/2}\brac{\abs{\tfrac{|S|}{N^{1/2}}-1}+N^{-1/6}}^{1/2}.
    \end{equation}
\end{theorem}
\begin{remark}
The exponent $-1/6$ appearing in \eqref{eq:fourier-unif} can be improved to $-1/4$. This is accomplished by replacing a use of \eqref{eq:SidonSize} in our proof with a sharper estimate of Cilleruelo \cite{CillerueloGaps}; see Theorem \ref{thm:improve-fourier-unif}.
\end{remark}
We note that on combining the Erd\H{o}s-Tur\'an upper bound \eqref{eq:SidonSize} with Singer's lower bound \eqref{eq:Singer}, the largest Sidon subset $S$ of $[N]$ satisfies $\abs{|S|-N^{1/2}} \ll N^{21/80}$, in which case the $N^{-1/6}$ error term dominates \eqref{eq:fourier-unif}. 
\begin{corollary}
The largest Sidon subset $S \subset [N]$ satisfies
\begin{equation}\label{eq:fourier-unif-cor}
       \bignorm{\hat{1}_{S} - \tfrac{|S|}{N}\hat{1}_{[N]}}_{\infty}
        \ll  \bignorm{\hat{1}_{S} }_{\infty}N^{ - \frac{1}{12}}.
\end{equation}
\end{corollary}

We are not the first to investigate the uniformity of extremal Sidon sets. Erd\H{o}s and Freud \cite{ErdosFreudSums} established that such sets are equidistributed in short intervals\footnote{There have since been quantitative improvements in this result, see \cite{CillerueloGaps}.}, whilst Lindstr\"om \cite{LindstromWell} proved equidistribution in arithmetic progressions\footnote{For quantitative improvements, see \cite{KolountzakisUniform}.}. We are able to re-prove (quantitatively weaker) versions of these results as a consequence of Theorem \ref{thm:sidon_pseudorandom}.
\begin{corollary}[Equidistribution in short intervals]\label{cor:EquiIntervals}
Let $I \subset [0,1]$ be an interval and $S \subset [N]$ a Sidon set of size\footnote{One could replace the factor $1/100$ with any positive absolute constant. This assumption makes our conclusions notationally simpler, and is always satisfied in the range of interest, when $|S| = N^{1/2}(1+o(1))$ with $N$ large.} $|S| \geq \recip{100} N^{1/2}$. Then we have the asymptotic\footnote{See \eqref{eq:expectation} for the definition of  $\E$.}
\begin{equation}\label{eq:EquiIntervals}
\E_{x \in S} 1_{I}(x/N)  = \meas(I) + 
O_\eps\brac{N^{\eps}\brac{\abs{\tfrac{|S|}{N^{1/2}}-1} + N^{-1/6}}^{1/2}.}
\end{equation}
\end{corollary}
\begin{corollary}[Equidistribution in residue classes]\label{cor:EquiAPs}
For any congruence class $a\pmod q$ and Sidon set $S \subset [N]$ of size $|S| \geq \recip{100} N^{1/2}$ we have the asymptotic
\begin{equation}\label{eq:EquiAPs}
\E_{x \in S} 1_{q\cdot \Z+a}(x)  = q^{-1} + O_\eps\brac{N^{\eps}\brac{\abs{\tfrac{|S|}{N^{1/2}}-1} + N^{-1/6}}^{1/2}.}
\end{equation}
\end{corollary}
Theorem \ref{thm:sidon_pseudorandom} is more general than Corollaries \ref{cor:EquiIntervals} and \ref{cor:EquiAPs}, yielding equidistribution in a wider class of sets. In the following, we write $\T := \R/\Z$ and $\norm{\alpha}_\T:= \min_{n \in \Z} |\alpha - n|$.
\begin{corollary}[Equidistribution in smooth Bohr neighbourhoods]\label{cor:BohrEqui}
Let $F : \T^d \to [0,1]$ have Lipschitz constant $K\geq 1$, in that for any $\alpha, \beta \in \T^d$ we have $$|F(\alpha) - F(\beta)| \leq K\max_j\norm{\alpha_j-\beta_j}_{\T}.$$ Then for any Sidon set $S \subset [N]$ of size $|S| \geq \recip{100} N^{1/2}$ and $\alpha \in \T^d$ we have
$$
\E_{x\in S} F(\alpha x) = \E_{x\in [N]} F(\alpha x) + O\brac{K^{\frac{2}{3}}\brac{\abs{\tfrac{|S|}{N^{1/2}}-1} + N^{-1/6}}^{\frac{1}{8d}}}.
$$
\end{corollary}
\begin{remark} 
We have not striven for quantitative efficiency in the error term of Corollary \ref{cor:BohrEqui}, which could be easily  improved.
\end{remark}
We can drop the smoothness assumption on the Bohr neighbourhood if we are prepared to assume that it is \emph{regular}.

\begin{definition}[Regular Bohr set]\label{def:reg-bohr}
Given $\alpha \in \T^d$ and $\rho >0$, we say that the Bohr set
$$
B(\alpha, \rho) := \bigset{x \in [N] : \max_i\norm{\alpha_i x} \leq \rho}
$$
is \emph{regular} if for any $|\kappa| \leq \frac{1}{100d}$ we have
$$
\abs{\frac{\abs{B\brac{\alpha, (1+\kappa)\rho}}}{|B\brac{\alpha, \rho}|} -1} \leq 100 d|\kappa|.
$$
\end{definition}
\begin{remark}
Bourgain \cite{BourgainTriples} established that regular Bohr sets are ubiquitous; see Tao and Vu \cite[Lemma 4.25]{TaoVuAdditive}.
\end{remark}
\begin{corollary}[Equidistribution in regular Bohr sets]\label{cor:EquiBohr}
Let $B = B(\alpha, \rho)$ be a regular Bohr set with $\alpha \in \T^d$. Then for any Sidon set $S \subset [N]$ of size
$|S| \geq \recip{100} N^{1/2}$ 
 we have
$$
\E_{x\in S} 1_{B}(x) = \E_{x\in [N]} 1_{B}(x) + O\brac{d\rho^{-1}\brac{\abs{\tfrac{|S|}{N^{1/2}}-1} + N^{-1/6}}^{\frac{1}{14d}}}.
$$
\end{corollary}
 \begin{remark} 
 As in Corollary \ref{cor:EquiAPs}, the error term in Corollary \ref{cor:EquiBohr} can be improved with a more refined analysis.
 \end{remark}
 
 \subsection{Partition regularity over extremal Sidon sets}
 
  We offer a further application of Theorem \ref{thm:sidon_pseudorandom} in proving a colouring analogue of a recent result of Conlon, Fox, Sudakov and Zhao \cite{CFSZRegularity}, this being the original motivation for our paper.
 
 Informally, call a Sidon subset of $[N]$ \emph{dense} if its cardinality is a
 positive proportion of $N^{1/2}$, say $\recip{100} N^{1/2}$. The authors of
 \cite{CFSZRegularity} show that any dense Sidon subset of $[N]$ contains a
 non-trivial\footnote{For instance, with all variables distinct.} solution to the equation $x_1+ x_2 + x_3 + x_4 = 4x_5$. The essential features of this equation are that its coefficients sum to zero and that it has at least five variables. The five variable condition cannot be relaxed: every Sidon set lacks non-trivial solutions to the four-variable equation $x_1-x_2-x_3+x_4 = 0$. The assumption that the coefficients sum to zero is also necessary, as we now show.
 \begin{proposition}\label{prop:bad-sidon}
 There exists a Sidon subset of $[N]$ with at least\\ $\frac{1}{\sqrt{2}} N^{1/2}(1+o(1))$ elements, and which has no solutions to the equation $x_1 + x_2 + x_3 + x_4 = x_5$
 \end{proposition}
 \begin{proof}
 Take an extremal Sidon subset $S_0$ of $[\frac{N-1}{2}]$ and set $S := 2\cdot S_0 +1$.
 \end{proof}
 \begin{remark}
 The above construction can be adapted to show that for any homogeneous linear equation whose coefficients do not sum to zero there exists a dense Sidon sets lacking solutions to the equation.
 \end{remark}
To accommodate equations whose coefficients do not sum to zero, one may speculate on whether a colouring analogue of the results of Conlon, Fox, Sudakov and Zhao \cite{CFSZRegularity} should hold; cf.\ Rado's criterion for partition regularity \cite[\S3.2]{GRSRamsey} versus Roth's criterion for density regularity \cite{RothCertainII}. Proposition \ref{prop:bad-sidon} indicates that one cannot hope to always find monochromatic solutions to $x_1 + x_2 + x_3 + x_4 = x_5$ in colourings of \emph{dense} Sidon sets. In the following theorem we show that such a result does hold for colourings of \emph{extremal} Sidon sets.
\begin{theorem}[Partition regularity over extremal Sidon sets]\label{thm:partition-regularity}\leavevmode
 Let $c_1, \dots,$ $c_s \in \Z\setminus\set{0}$ with $s \geq 5$ and suppose that there exists a non-empty index set $I \subset [s]$ satisfying $\sum_{i \in I} c_i = 0$. Let $r$ be a positive integer and $S \subset [N]$ a Sidon set. Then at least one of the following holds:
 \begin{itemize}
 \item  $N$ is small, in that $N \ll_{c_1, \dots, c_s, r} 1$.
 \item  $S$ is not extremal, in that $\abs{|S|-N^{1/2}} \gg_{c_1,\dots, c_s, r} N^{1/2}$.
 \item Partition regularity: For any $r$-colouring $S= C_1 \cup\dots \cup C_r$, there exists a colour class $ C_j$ such that
$$
\sum_{c_1x_1 + \dots + c_sx_s = 0} 1_{C_j}(x_1) \dotsm 1_{C_j}(x_s) \gg_{c_1,\dots, c_s, r} |S|^{s}N^{-1}
$$
\end{itemize}
\end{theorem}
\begin{remark}
If $\sum_{i \in I} c_i \neq 0$ for all $\emptyset \neq I \subset [s]$, then there exists a finite colouring of $\N$ with no monochromatic solutions to the equation $c_1x_1+\dots + c_s x_s = 0$ (see Rado's criterion for partition regularity \cite[\S3.2]{GRSRamsey}). Hence the assumption that some subset of coefficients sums to zero is necessary.
\end{remark}
Conlon, Fox, Sudakov and Zhao \cite{CFSZRegularity} derive their results on dense Sidon sets via a removal lemma for $C_4$-free graphs. Since ``extremal'' $C_4$-free graphs are pseudorandom, see \cite[Theorem 5.1]{ksv13}, the transference approach employed in this paper can surely be combined with a counting lemma from \cite[Theorem 3.1]{CFSZRegularity}, to prove that every finite colouring of an ``extremal'' $C_4$-free graph has a monochromatic $C_5$. It may be interesting to investigate which other monochromatic subgraphs can be guaranteed in this manner, see \cite{CFSZWhich}.

\subsection*{Acknowledgements}
We thank David Conlon for an informative talk on \cite{CFSZRegularity}  in the \emph{Webinar in Additive Combinatorics}\footnote{\url{https://sites.google.com/view/web-add-comb/}}.

\subsection*{Notation}
\subsubsection*{Standard conventions}
We use $[N]$ to denote the set of integers $ \{ 1,2, \dots, N\}$.  We use
counting measure on $\Z$, so that for $f,g :\Z \to \C$, we have
$$
\norm{f}_{p} := \biggbrac{\sum_x |f(x)|^p}^{\recip{p}}\ \text{and}\ (f*g)(x) := \sum_y f(y)g(x-y).
$$ 
Any sum of the form $\sum_x$ is to be interpreted as a sum over $\Z$. The \emph{support} of $f$ is the set $\supp(f) := \set{x \in \Z : f(x) \neq 0}$. 
For a finite set $S$ and function $f:S\to\C$, denote the average of $f$ over $S$ by
\begin{equation}\label{eq:expectation}
\E_{s\in S}f(s):=\frac{1}{|S|}\sum_{s\in S}f(s).
\end{equation}
 
We use Haar probability measure on $\T := \R/\Z$, so that for integrable $F : \T \to \C$, we have
\begin{equation*}
\norm{F}_{p} := \biggbrac{\int_\T |F(\alpha)|^p\intd\alpha}^{\recip{p}} = \biggbrac{\int_0^1 |F(\alpha)|^p\intd\alpha}^{\recip{p}}
\end{equation*}
and
\begin{equation*}
\norm{F}_{\infty} := \sup_{\alpha \in \T} |F(\alpha)|.
\end{equation*}
Write $\norm{\alpha}_\T:= \min_{n \in \Z} |\alpha - n|$ for the distance from $\alpha \in \R$ to the nearest integer. This remains well-defined on $\T$.

\subsubsection*{Asymptotic notation}
For a complex-valued function $f$ and positive-valued function $g$, write $f \ll g$ or $f = O(g)$ if there exists a constant $C$ such that $|f(x)| \le C g(x)$ for all $x$. We write $f = \Omega(g)$ if $f \gg g$.  The notation $f\asymp g$ means that $f\ll g$ and $f\gg g$. We subscript these symbols if the implicit constant depends on additional parameters.

 We write $f = o(g)$ if for any $\eps > 0$ there exists $X\in \R$ such that for all $x \geq X$ we have $|f(x)| \leq \eps g(x)$.

\subsubsection*{Local conventions}
Up to normalisation, all of the above are widely used in the literature. Next, we list notation specific to our paper. We have tried to minimise this in order to aid the casual reader.  

For a real parameter $H \geq 1$, we use $\mu_H : \Z \to [0,1]$ to represent the following normalised Fej\'er kernel
\begin{equation}\label{fejer}
\mu_H(h) := \recip{\floor{H}} \brac{1 - \frac{|h|}{\floor{H}}}_+ = \frac{(1_{[H]} * 1_{-[H]} )(h)}{\floor{H}^2},
\end{equation}
where $[H] = [\lfloor{H}\rfloor]$.
This is a probability measure on $\Z$ with support in the interval $(-H, H)$. 

\section{Fourier uniformity}
Given an extremal Sidon set $S \subset [N]$, our main goal in this section is to 
show that $S$ is a Fourier uniform subset of $[N]$, by proving Theorem \ref{thm:sidon_pseudorandom}. Recalling our notation \eqref{fejer} for the Fej\'er kernel, we begin with the following version of van der Corput's inequality.
\begin{lemma}[Van der Corput differencing]
    Suppose that $1 \leq H \leq N$, $f \colon \Z \to \C$ and $\supp(f) \subset [N]$. Then
    \begin{equation}
        \label{eq:vandercorput}
        \biggabs{\sum_x f(x) }^2 \leq (N+H) \sum_{h} \mu_H(h) \sum_x f(x)
        \overline{f(x+h)}.
    \end{equation}
\end{lemma}
\begin{proof}
    We have
    \[
        \biggabs{\sum_x f(x)}^2 = \biggabs{\E_{[H]} \sum_x f(x+h)}^2 =
       \biggabs{\sum_x \E_{[H]}f(x+h)}^2.
    \]
    By Cauchy-Schwarz, the latter quantity is bounded by
    \[
        (N+H) \sum_x \biggabs{\E_{[H]} f(x+h)}^2 = (N+H)\sum_x
        \frac{1}{\lfloor H \rfloor^2}\sum_{h_1, h_2 \in [H]} f(x+h_1) \overline{f(x+h_2)}.
    \]
    We obtain the desired
    inequality on changing variables in $x$ and using \eqref{fejer}.
\end{proof}

The next lemma tells us that, on taking an
appropriately sized $H$ and supposing that $S$ is extremal Sidon, the sum of $\mu_H(h)$
over $S-S$ is nearly $1$.
\begin{lemma}
    \label{lem:sidondif_sum_lower_bound}
    Let $S \subset [N]$ be a Sidon set. Then
     \[
        \sum_{h \in (S-S)\setminus\set{0}} \mu_H(h) \geq \frac{|S|^2}{N+H} - \frac{|S|}{\floor{H}}.
            \]
\end{lemma}
\begin{proof}
    Since $S$ is Sidon $1_S*1_{-S}(x) = 1_{S-S}$ if $x \neq
    0$. In addition $1_S*1_{-S}(0) = |S|$, but this does not require $S$ to be Sidon. Hence
    \begin{equation}
        \label{eq:sidondif_sum}
        \sum_{h \in (S-S)\setminus\set{0}} \mu_H(h) = \sum_{h} 1_S*1_{-S}(h)\mu_H(h) - \frac{|S|}{\floor{H}}.
    \end{equation}
    Using \eqref{fejer} and expanding convolutions we have 
    \begin{multline*}
        \sum_h 1_S*1_{-S}(h)\mu_H(h) = \frac{1}{\floor{H}^2}\sum_h
        1_S*1_{-S}(h)1_{[H]}*1_{-[H]}(h)\\ = \frac{1}{\floor{H}^2}\sum_h 1_{S}*1_{[H]}(h)^2.
    \end{multline*}
    We may now apply Cauchy-Schwarz to this last sum, giving us
    \[
        \sum_h 1_S*1_{-S}(h)\mu_H(h)  \geq \frac{1}{(N+H)\floor{H}^2}\left(\sum_h
        1_{S}*1_{[H]}(h)\right)^2 = \frac{|S|^2}{(N+H)}, 
    \]
    which gives the claimed inequality.
\end{proof}
We are now in a position to prove Theorem \ref{thm:sidon_pseudorandom}.
Our use of van der Corput's inequality is reminiscent of
the moving averages argument used by Erd\H{o}s and Tur\'an \cite{ErdosTuranProblem}.
\begin{proof}[Proof of Theorem \ref{thm:sidon_pseudorandom}]
    Let $f_1 := 1_S$, $f_2 := \frac{|S|}{N}1_{[N]}$ and $f :=  f_1-f_2$. Applying van der Corput's inequality
    \eqref{eq:vandercorput} to $\hat{f}$, with $1\leq H \leq N$ to be determined, we obtain
    \begin{multline*}
        |\hat{f}(\alpha)|^2 \leq 
        (N+H) \sum_{h} \mu_H(h)\\
        \times \left|\sum_{x}
        \bigsqbrac{f_1(x)f_1(x+h) - f_1(x)f_2(x+h) - f_2(x)f_1(x+h) + f_2(x)f_2(x+h)} \right|.
    \end{multline*}
       We claim that $\sum_{h} \mu_H(h)| |S|^2N^{-1} - \sum_xf_i(x)f_j(x+h)|$ is
    small for all choices of $i, j$, so that main terms cancel  and we are left only with  error
    terms. 
    
    Since $f_2(x)f_2(x+h) = |S|^2N^{-2}1_{[N] \cap
    ([N]-h)}(x)$ we have
    \[
        \sum_{h} \mu_H(h) \left| |S|^2N^{-1}- \sum_xf_2(x)f_2(x+h)\right| =
       \frac{|S|^2}{N^{2}} \sum_{h} \mu_H(h) |h|\leq \frac{H|S|^2}{N^2}.
    \]
    We have the identity $f_1(x)f_2(x+h) = |S|N^{-1} 1_{S \cap ([N]-h)}(x) = |S|N^{-1} [1_{S }(x) - 1_{S\cap I_h}(x)]$, where $I_h \subset [N]$ is an interval of $|h|$ integers.  Since $S$ is a Sidon set, the bound \eqref{eq:SidonSize} gives that 
    \begin{equation}\label{eq:weak-part}
    |S\cap I_h|\ll \sqrt{|h|}.
    \end{equation}
     Hence
    \begin{multline*}
       \sum_h \mu_H(h) \left||S|^2N^{-1} - \sum_xf_1(x)f_2(x+h)\right| \ll \sum_h \mu_H(h)|S|N^{-1}|h|^{1/2}\\ \ll \frac{|S|H^{1/2}}{N}.
    \end{multline*}
  By symmetry, the same bound applies to the term involving $f_2(x)f_1(x+h)$.
    
    Note that
    $ f_1(x)f_1(x+h) =  1_{S \cap (S-h)}(x)$. On account of $S$
    being Sidon, for $h \neq 0$ we have
    \[
    \begin{split}
       \abs{ |S|^2N^{-1} - \sum_xf_1(x)f_1(x+h)} & \leq    1 - 1_{S-S}(h)+ \abs{|S|^2-N}N^{-1}\\
       & = 1 - 1_{S-S}(h) + O\brac{\abs{|S|- N^{1/2}}N^{-1/2}}.
       \end{split}
    \]
    Thus by Lemma \ref{lem:sidondif_sum_lower_bound} we obtain
        \begin{multline*}
            \sum_{h}  \mu_H(h) \left| |S|^2N^{-1} - \sum_xf_1(x)f_1(x+h)\right| 
            \leq 1 -\frac{|S|^2}{N+H} \\
            + O\brac{\frac{|S|}{\floor{H}}+\frac{\abs{|S|- N^{1/2}}}{N^{1/2}}}.
        \end{multline*}
    Putting everything together, we deduce that
     \[
            |\hat{f}(\alpha)|^2 \ll 
        N^{1/2}\abs{|S|- N^{1/2}}+ \frac{N^{3/2}}{\floor{H}}+ N^{1/2}H^{1/2}.
        \]
Balancing error terms, we set $H :=  N^{2/3}
    $ which gives
   
    \[
        |\hat{f}(\alpha)|^2 \ll N^{1/2}\abs{|S|- N^{1/2}} + N^{5/6}. \qedhere
    \]
\end{proof}
\section{Equidistribution in progressions}
In this section we derive Corollaries \ref{cor:EquiIntervals} and \ref{cor:EquiAPs} from Theorem \ref{thm:sidon_pseudorandom}. Both corollaries are immediate consequences of the following.
\begin{theorem}\label{thm:equi-prog}
  Let $S\subset [N]$ be a Sidon set 
  and $P \subset \Z$
  a finite arithmetic progression. Then
  \[
  |S\cap P| = \frac{|[N]\cap P||S|}{N} + O_\eps\brac{N^{1/2+\eps}\brac{\abs{\tfrac{|S|}{N^{1/2}}-1} + N^{-1/6}}^{1/2}}.
  \]
\end{theorem}

Theorem \ref{thm:equi-prog} follows from Theorem \ref{thm:sidon_pseudorandom} by the Erd\H{o}s-Tur\'an inequality (see, for example, \cite[Corollary 1.1]{MontgomeryTen}). We prove a cheap version of this inequality that suffices for our purposes. We begin with a standard  estimate for the $L^1$ norm of the Fourier transform of a progression.
\begin{lemma}
  \label{lem:bound_interval}
  Let $P \subset \Z$ be an arithmetic progression. Then
  \[
    \int_\T \left|\hat{1}_P(\alpha)\right| \intd \alpha \ll \log(|P|).
  \]
\end{lemma}
\begin{proof}
  With a suitable change of variables, we may assume that $P = \{1, \dots,
  |P|\}$. Let us first prove a pointwise bound. Summing the
  geometric series gives
  \[
    \left|\hat{1}_P(\alpha)\right| \ll \frac{1}{|1-e(\alpha)|} = \frac{1}{|e(-\alpha/2) -
    e(\alpha/2)|} = \frac{1}{|2\sin(\pi \alpha)|} \ll \frac{1}{\|\alpha\|}, 
  \]
  where we have used that $2|\alpha| \leq |\sin(\pi\alpha)|$ for $\alpha \in [-1/2,
  1/2]$. Thus,
  \[
    \left|\hat{1}_P(\alpha)\right| \ll \min\left(|P|, \frac{1}{\|\alpha\|}\right).
  \]
  
  We now use dyadic decomposition to complete the proof. Let
  \[
    A_k = \left\{\alpha \in \T \colon \frac{2^{k-1}}{|P|} \leq \|\alpha\| \leq
    \frac{2^{k}}{|P|}\right\}
  \]
  for $k = \left\{ 1, \dots, \lceil\log_2(|P|) \rceil \right\}$ and $A_0 = \left\{
  \alpha \in \T \colon \|\alpha\| \leq 1/|P| \right\}$. For $k \geq 1$ we have
  \[
    \int_{A_k} \left|\hat{1}_P(\alpha)\right| \intd \alpha \leq  \int_{A_k}
    \frac{1}{\|\alpha\|} \intd \alpha \leq \frac{|P|\meas(A_k)}{2^{k-1}} \ll 1.
  \]
  On the other hand,
  \[
    \int_{A_0}\left|\hat{1}_P(\alpha)\right| \intd \alpha \ll |A_0||P| \ll 1.
  \]
  Since $\T \subset \bigcup_k A_k$, the claimed bound follows.
\end{proof}
We are now able to deduce Theorem \ref{thm:equi-prog} from Theorem \ref{thm:sidon_pseudorandom}.
\begin{proof}[Proof of Theorem \ref{thm:equi-prog}]
  Using orthogonality we have
  \begin{align*}
     |S\cap P| - \tfrac{|[N]\cap P||S|}{N} &= \sum_x 1_{P\cap [N]}(x)\brac{1_{S}(x) -
    \tfrac{|S|}{N}1_{[N]}(x)}\\
    &= \int_{\T} \hat{1}_{P\cap [N]}(\alpha) \brac{\hat{1}_{S} - \tfrac{|S|}{N}\hat{1}_{[N]}}(-\alpha)\intd\alpha.
  \end{align*}
  The result follows from Lemma \ref{lem:bound_interval} and
  the Fourier uniformity obtained in Theorem \ref{thm:sidon_pseudorandom}.
\end{proof}

\section{Equidistribution in Bohr neighbourhoods}

One can prove Corollaries \ref{cor:BohrEqui} and \ref{cor:EquiBohr} using Erd\H{o}s--Tur\'an type arguments, see for instance \cite[Chapter 1]{MontgomeryTen}.  We opt for the following cruder trigonometric approximation, a proof of which can be found in Appendix \ref{app:trig-approx}.

\begin{lemma}[Trigonometric approximation]\label{lem:trig_approx}
Let $F : \T^d \to [0,1]$ have Lipschitz constant $K\geq 1$ with respect to the metric 
\begin{equation}\label{eq:TMetric}
\max_j \norm{\alpha_j - \beta_j}_{\T}.
\end{equation}
For any $\eps > 0$ there exists a trigonometric polynomial $F_\eps : \T^d \to [0,1]$ with $\norm{F-F_\eps}_\infty \leq \eps$ and such that 
$$
F_\eps(\alpha) = \sum_{|m_i|\leq M} \hat{F_{\varepsilon}}(m)e(m \cdot \alpha)
$$
with $M \ll K^2\eps^{-3}$.
\end{lemma}

\begin{proof}[Deduction of Corollary \ref{cor:BohrEqui}]
Let $F : \T^d \to [0,1]$ be a 1-bounded function with Lipschitz constant $K\geq 1$ with respect to the metric \eqref{eq:TMetric}. 
We apply Lemma \ref{lem:trig_approx} to obtain a trigonometric approximation $F_\eps : \T^d \to [0,1]$, with $\eps >0$ to be determined. 

Expanding $F_\eps$ in terms of its Fourier coefficients gives
$$
\sum_{x \in S} F_\eps(\alpha x) =  \sum_{ |m_i| \leq M} \hat{F_{\varepsilon}}(m) \sum_{x \in S} e(m\cdot \alpha x).
$$
Approximating $\sum_{x\in S}e(m\cdot \alpha x)$ with $\sum_{x\in [N]}e(m\cdot \alpha x)$, we deduce that there exists an absolute constant $C$ such that 
\begin{equation*}
\biggabs{\sum_{x \in S} F_\eps(\alpha x) -  \tfrac{|S|}{N}\sum_{x \in [N]} F_\eps(\alpha x)}
\leq (CK^2/\eps^{-3})^{d} \bignorm{\hat{1}_S - \tfrac{|S|}{N} \hat{1}_N }_{\infty}.
\end{equation*}
Balancing the above error term with $\eps |S|$, we  take
$$
\eps := \brac{|S|^{-1}\brac{ CK^2}^{d} \bignorm{\hat{1}_S - \tfrac{|S|}{N} \hat{1}_N }_{\infty}}^{\frac{1}{3d+1}}
$$
to yield the asymptotic
\begin{equation*}
\abs{\E_{x \in S} F(\alpha x) - \E_{x \in [N]} F(\alpha x)}\\ \ll
\brac{K^{2d}\brac{\abs{\tfrac{|S|}{N^{1/2}}-1} + N^{-1/6}}^{1/2}}^{\frac{1}{3d+1}},
\end{equation*}
on employing Theorem \ref{thm:sidon_pseudorandom}.
\end{proof}

\begin{proof}[Deduction of Corollary \ref{cor:EquiBohr}]
Let $\eps > 0$ be a small quantity to be determined. Let $F_1 : \T \to [0,1]$ be a piecewise linear function with $F_1(\alpha) = 1$ on $[-\rho, \rho]$, $F_1(\alpha) = 0$ on $\T\setminus [-\rho-\eps,\rho+\eps]$ and with Lipschitz constant at most $\eps^{-1}$. By a telescoping identity, the function $F(\alpha) := F_1(\alpha_1) \dotsm F_1(\alpha_d)$ has Lipschitz constant at most $d\eps^{-1}$. Hence by Corollary \ref{cor:BohrEqui} we have
\begin{multline*}
\E_{x \in S} 1_{B}(x) \leq \E_{x \in S} F\brac{\alpha x}\\
=  \E_{x \in [N]} F\brac{\alpha x}+ O\brac{(d/\eps)^{\frac{2}{3}}\brac{\abs{\tfrac{|S|}{N^{1/2}}-1} + N^{-1/6}}^{\frac{1}{8d}}}.\end{multline*}

Since $F$ is supported on $B(\alpha, \rho + \eps)$, regularity (Definition \ref{def:reg-bohr}) ensures that
$$
\sum_{x \in [N]} F\brac{\alpha x} \leq |B| + O\brac{ d\rho^{-1} \eps N}.
$$
Therefore
\begin{equation*}
\E_{x \in S} 1_{B}(x) \leq \tfrac{|B|}{N}\\
+   O\brac{d\rho^{-1} \eps +(d/\eps)^{\frac{2}{3}}\brac{\abs{\tfrac{|S|}{N^{1/2}}-1} + N^{-1/6}}^{\frac{1}{8d}}}.
\end{equation*}
Choosing $\eps$ to balance error terms, then bounding exponents somewhat crudely, we obtain
\begin{equation*}
\E_{x \in S} 1_{B}(x) \leq \tfrac{|B|}{N}\\
+   O\brac{d\rho^{-1}\brac{\abs{\tfrac{|S|}{N^{1/2}}-1} + N^{-1/6}}^{\frac{1}{14d}}}.
\end{equation*}

The corresponding lower bound is proved analogously. 
\end{proof}

\section{Partition regularity}
Our deduction of Theorem \ref{thm:partition-regularity} from Theorem \ref{thm:sidon_pseudorandom} requires two additional results, the first being the following transference principle for colourings, a proof of which can be found in Appendix \ref{app:dense-model}.
\begin{lemma}[Dense model lemma]\label{lem:dense-model}
Let $0< \eps \leq 1$ and $\nu : [N] \to [0, \infty)$ be such that there exist functions $f_i : [N]\to [0,\infty)$ with $f_1 + \dots + f_r \leq \nu$ and such that for any $g_i :[N] \to [0, \infty)$ satisfying $g_1 + \dots + g_r \leq 1_{[N]}$ there exists $i$ with
$$
\bignorm{\hat{f}_i - \hat{g}_i}_\infty > \eps N.
$$
Then 
\begin{equation}\label{eq:non-unif}
\bignorm{\hat{\nu}- \hat{1}_{[N]}}_\infty \gg_{\eps,r} N.
\end{equation}
\end{lemma}
The second additional result underlying Theorem \ref{thm:partition-regularity} is a lower bound on the number of monochromatic solutions to a partition regular equation in an interval.
\begin{lemma}[Counting monochromatic solutions in an interval of integers]\label{lem:FGR}
Let $c_1$, \dots, $c_s \in \Z\setminus\set{0}$ and suppose that there is a non-empty index set $I \subset [s]$ satisfying $\sum_{i \in I } c_i = 0$.  Then for any functions $g_1, \dots, g_r : [N] \to [0,\infty)$ with $1_{[N]} = g_1 + \dots + g_r$, either $N\ll_{c_1, \dots, c_s, r} 1$ or there exists $g_j$ satisfying 
\begin{equation}\label{eq:FGR}
\sum_{c_1x_1 + \dots + c_sx_s = 0} g_j(x_1) \dotsm g_j(x_s) \gg_{c_1, \dots, c_s, r} N^{s-1}.
\end{equation}
\end{lemma}

\begin{proof}
When each $g_i$ is a characteristic function of a set $C_i \subset [N]$, the result is a special case of Frankl, Graham and R\"odl \cite[Theorem 1]{FGRQuantitative}. In general, for each $x \in [N]$ fix an index $i = i(x)$ such that $g_i(x) \geq 1/r$. At least one such index exists by the pigeon-hole principle.  On setting $C_i = \set{x\in [N] : i(x) = i}$ we obtain a colouring, and the lower bound \eqref{eq:FGR} follows on employing \cite[Theorem 1]{FGRQuantitative}.
\end{proof}
\begin{proof}[Proof of Theorem \ref{thm:partition-regularity}]
Let $S \subset [N]$ be a Sidon set and $S = C_1 \cup \dots \cup C_r$.   Writing $\delta = \delta_{c_1,\dots, c_s, r} > 0$ for the implicit constant in \eqref{eq:FGR}, suppose that 
\begin{equation}\label{eq:nonFGR}
\sum_{c_1x_1 + \dots + c_sx_s = 0} 1_{C_j}(x_1) \dotsm 1_{C_j}(x_s) \leq \trecip{2}\delta |S|^{s}N^{-1}\qquad (1\leq j \leq r).
\end{equation}
Define $f_i := N|S|^{-1} 1_{C_i}$, so that $f_1 + \dots + f_{r} = N|S|^{-1} 1_S$.  Then, by Lemma \ref{lem:FGR}, we either have $N \ll_{c_1,\dots, c_s, r} 1$, or for any functions $g_j \geq 0$ with $g_1 + \dots + g_r = 1_{[N]}$   there exists $g_j$ such that
\begin{multline}\label{eq:count-diff}
\abs{\sum_{c_1x_1 + \dots + c_sx_s = 0} \sqbrac{g_j(x_1) \dotsm g_j(x_s)- f_j(x_1) \dotsm f_j(x_s)}}\gg_{c_1, \dots, c_s, r} N^{s-1}.
\end{multline}
Given \eqref{eq:count-diff}, by a telescoping identity, there exists $h_1, \dots, h_s \in \set{g_j, f_j, f_j- g_j}$, exactly one of which is equal to $f_j-g_j$, and such that
\begin{multline}
\abs{\sum_{c_1x_1 + \dots + c_sx_s = 0} \sqbrac{g_j(x_1) \dotsm g_j(x_s)- f_j(x_1) \dotsm f_j(x_s)}}\\ \ll_s \abs{\sum_{c_1x_1 + \dots + c_sx_s = 0} h_1(x_1) \dotsm h_s(x_s)}.
\end{multline}
By orthogonality and H\"older's inequality
\begin{multline}\label{eq:holder}
 \abs{\sum_{c_1x_1 + \dots + c_sx_s = 0} h_1(x_1) \dotsm h_s(x_s)} 
 = \abs{ \int_\T \hat{h}_1(c_1\alpha) \dotsm \hat{h}_s(c_s\alpha) \intd\alpha}\\ \leq \bignorm{\hat f_j - \hat g_j}_\infty \max\set{\bignorm{\hat f_j}_{s-1}, \bignorm{\hat g_j}_{s-1}}^{s-1}.
\end{multline}
Since $s-1 \geq 2$ and $0\leq g_j \leq 1_{[N]}$, Parseval's identity gives that
$$
\bignorm{\hat g_j}_{s-1}^{s-1} \leq N^{s-2}.
$$
Since $s-1 \geq 4$ and $0 \leq f_j \leq N|S|^{-1} 1_S$, orthogonality and the Sidon property give that 
$$
\bignorm{\hat f_j}_{s-1}^{s-1} \leq N^{s-1} |S|^{-4} \sum_{x-x' = y-y'} 1_S(x)1_S(x') 1_S(y) 1_S(y') \leq 2N^{s-1}|S|^{-2}.
$$
Supposing that $|S| \geq \recip{100}N^{1/2}$, the latter quantity is $O(N^{s-2})$. We may assume that $|S| \geq \recip{100}N^{1/2}$ for otherwise $||S|-N^{1/2}| \gg N^{1/2}$.

From the above deliberations, we conclude that if \eqref{eq:nonFGR} holds, then either $N \ll_{c_1, \dots, c_s, r} 1$, or $||S|-N^{1/2}| \gg N^{1/2}$,  or for any $g_1, \dots, g_r \geq 0$ with $g_1 + \dots + g_r = 1_{[N]}$ there exists $g_j$ such that 
\begin{equation}\label{eq:no-model}
\bignorm{\hat f_j - \hat g_j}_\infty \gg_{c_1, \dots, c_s, r} N.
\end{equation}
Henceforth we assume that we are not in the situation that $N \ll_{c_1, \dots, c_s, r} 1$ or $||S|-N^{1/2}| \gg N^{1/2}$. Let $\eta = \eta(c_1,\dots, c_r, r) > 0$ denote the implicit constant in \eqref{eq:no-model}. If it is the case that there exists $g_1, \dots, g_r \geq 0$ with $g_1 + \dots + g_r = 1_{[N]}$ such that for all $1\leq j \leq r-1$ we have
$$
\bignorm{\hat f_j - \hat g_j}_\infty \leq \tfrac{\eta}{2r}  N,
$$
then \eqref{eq:no-model} holds with $j = r$, so by the triangle inequality 
$$
\bignorm{N|S|^{-1}\hat 1_S - \hat 1_{[N]}}_\infty \geq \trecip{2} \eta N \gg_{c_1, \dots, c_s, r} N.
$$
Let us show that this conclusion also holds when for any $g_1, \dots, g_r \geq 0$ with $g_1 + \dots + g_r = 1_{[N]}$ there exists $1\leq j \leq r-1$ such that
$$
\bignorm{\hat f_j - \hat g_j}_\infty > \tfrac{\eta}{2r}  N.
$$
Since
\begin{multline*}
\set{(g_1, \dots, g_r) : g_1 + \dots + g_r = 1_{[N]} \text{ and } g_i \geq 0\text{ for all }i}\\
 = \set{(g_1, \dots, g_{r-1}, 1_{[N]}- g_r) : g_1 + \dots + g_{r-1} \leq 1_{[N]} \text{ and } g_i \geq 0\text{ for all }i},
\end{multline*}
 we may apply the dense model lemma (Lemma \ref{lem:dense-model}) to conclude that 
$$
\bignorm{N|S|^{-1}\hat 1_S - \hat 1_{[N]}}_\infty \gg_{c_1, \dots, c_s, r} N.
$$
Hence by Theorem \ref{thm:sidon_pseudorandom} we have
\begin{equation}\label{eq:final-ineq}
N^{1/4}\abs{|S|-N^{1/2}}^{1/2} + N^{5/12} \gg_{c_1, \dots, c_s, r} |S|.
\end{equation}
Again assuming $|S| \geq \trecip{100}N^{1/2}$ (as we may), \eqref{eq:final-ineq} implies that either $N \ll_{c_1, \dots, c_s, r} 1$ or $\abs{|S|-N^{1/2}} \gg_{c_1, \dots, c_s, r} N^{1/2}$.
\end{proof}

\section{Improving Fourier uniformity}\label{sec:improving-fourier}

In  previous sections we have seen that the quality of Fourier uniformity dictates the level of equidistribution of an extremal Sidon set. In this section we give a modified proof of Theorem \ref{thm:sidon_pseudorandom} with an increased power saving in the quality of Fourier uniformity. This is accomplished by incorporating an estimate of Cilleruelo \cite[Theorem 1.1]{CillerueloGaps} on the level of equidistribution in short intervals.
\begin{theorem}[Cilleruelo]\label{thm:uniform_distribution}
  Let $S \subset [N]$ be a Sidon set and 
   $I \subset [N]$  an interval. Then
  \[
    \abs{ |S \cap I| - \frac{|I||S|}{N} } \ll \brac{N^{1/4} + |I|^{1/2}
    N^{-1/8}}\brac{1+\brac{1-\tfrac{|S|}{N^{1/2}}}_{+}^{1/2} N^{1/8}},
  \]
  where $x_+ = \max(0, x)$.
\end{theorem}

\begin{corollary}
  \label{cor:uniform_distribution}
  Let $S \subset [N]$ be a Sidon set and
   $I \subset [N]$ an interval with $|I| \leq  N^{3/4}$. Then
  \[
    |S \cap I| \ll N^{1/4} + \abs{1-\tfrac{|S|}{N^{1/2}}}^{1/2}N^{3/8}.
  \]
\end{corollary}
Using this result we can  refine the bounds obtained in
Theorem \ref{thm:sidon_pseudorandom}.
\begin{theorem}\label{thm:improve-fourier-unif}
  Let $S \subset [N]$ be a Sidon set. Then
    \begin{equation}\label{eq:improve-fourier-unif}
        \left\|\hat{1}_{S} - \tfrac{|S|}{N}\hat{1}_{[N]}\right\|_{\infty}
        \ll N^{1/2}\brac{\abs{1-\tfrac{|S|}{N^{1/2}}}+ 
        N^{-1/4}}^{1/2}.
    \end{equation}
    \end{theorem}
\begin{proof}
  The proof is identical to that given for Theorem
  \ref{thm:sidon_pseudorandom}, albeit taking $H := N^{3/4}$ and replacing our use of \eqref{eq:SidonSize} in \eqref{eq:weak-part} with Corollary \ref{cor:uniform_distribution}. This gives
  $$
  |S \cap I_h| \ll N^{1/4} + \abs{1-\tfrac{|S|}{N^{1/2}}}^{1/2}N^{3/8}.
  $$
  Hence when $i \neq j$ we have
    \begin{multline*}
       \sum_h \mu_H(h) \left||S|^2N^{-1} - \sum_xf_i(x)f_j(x+h)\right|\\ \ll  |S|\brac{N^{-3/4} + \abs{1-\tfrac{|S|}{N^{1/2}}}^{1/2}N^{-5/8}}.
    \end{multline*}

Putting everything together, as in the proof of Theorem \ref{thm:sidon_pseudorandom}, then gives the desired bound.
\end{proof}

\begin{remark}
This second version of our Fourier uniformity bound 
solves a quirk of the previous proof, where we took $H =  N^{2/3} $,
whereas in many results on extremal Sidon sets taking $H =  N^{3/4}
$ appears naturally.

Of course, this new version cannot be used to improve the bounds on uniform
distribution in intervals, since it would give a circular argument. However, it
may be applied to improve our bounds on distribution in residue classes.
\end{remark}


\appendix

\section{The size of a Sidon set}\label{sec:SidonSize}
In this appendix we establish the well-known bound \eqref{eq:SidonSize}, again
using van der Corput's variant of the Cauchy--Schwarz inequality.
Let $S \subset (n, n+N]$ be a Sidon set and $H$ be a positive integer (to be determined).  Applying
\eqref{eq:vandercorput} to $\sum_x 1_S(x)$ we obtain
\[
  |S|^2 \leq \brac{N+\floor{H}} \sum_h \mu_H(h) \sum_x 1_S(x) 1_S(x+h).
\]

Using the defining property of Sidon sets, and the fact that $\mu_H$ is a probability measure, we deduce that 
\begin{equation*}
  |S|^2 \leq  \brac{N+H}\brac{\frac{|S|}{H} + 1}.
\end{equation*}
By the quadratic formula $x^2 \leq bx + c$ only if $x \leq (b + \sqrt{b^2+4c})/2 $, which in turn implies that $x \leq b + \sqrt{c}$.
Hence
\[
|S| \leq \sqrt{N+H} + (N+H)H^{-1} \leq  N^{1/2} + HN^{-1/2} + NH^{-1} + 1.
\]
The bound \eqref{eq:SidonSize} follows on taking, say, $H = \ceil{N^{3/4}}$.

\section{Trigonometric approximation}\label{app:trig-approx}

\begin{definition}
  Given an integrable function $F \colon \T^d \to [0, 1]$, we define its Fourier
  transform to be the function $\hat{F} \colon \Z^d \to \C$ given by
  \[
    \hat{F}(m) = \int_{\alpha \in \T^d} F(\alpha) e(-m \cdot \alpha),
  \]
  where $m \cdot \alpha = m_1 \alpha_1 + \cdots + m_d \alpha_d$.
\end{definition}

\begin{proof}[Proof of Lemma \ref{lem:trig_approx}]
Let $\lambda_M(\alpha) = \lambda_M(\alpha_1) \dotsm \lambda_M(\alpha_d)$ denote the following renormalised Fourier transform of the F\'ejer kernel:
\begin{multline*}
  \sum_m \brac{1 - \frac{|m_1|}{M}}_{+}\dotsm \brac{1 - \frac{|m_d|}{M}}_{+} e(m\cdot\alpha)\\ = M^{-d} \brac{1_{[M]^d}*1_{-[M]^d}}\hat{\ } (\alpha) = M^{-d}\abs{\hat{1}_{[M]^d}(\alpha)}^2.
\end{multline*}
Set 
$$
F_M(\alpha) := F*\lambda_M(\alpha) =  \int_{\T^d} F(\alpha-\beta) \lambda_M(\beta)  \intd\beta.
$$  
One can check that
$$
F*\lambda_M(\alpha) = \sum_m \brac{1 - \frac{|m_1|}{M}}_{+}\dotsm \brac{1 - \frac{|m_d|}{M}}_{+} \hat{F}(m)e(m\cdot\alpha).
$$

We utilise the following three properties of the F\'ejer kernel.
\begin{enumerate}[{\normalfont (a)}]
\item  (Non-negativity) $\lambda_M \geq 0$.
\item  (Mass one) $\int_{\T^d} \lambda_M = 1$.
\item  (Quantitative decay) $\lambda_M(\alpha) \leq M^{-1}\norm{\alpha_j}^{-2} \prod_{i\neq j} \lambda_M(\alpha_i) $.
\end{enumerate}

The first two facts ensure that  $0 \leq F_M \leq 1$, since $0 \leq F \leq 1$. Let us estimate the error $\norm{F - F*\lambda_M}_{\infty}$.  By definition of convolution
\begin{equation*}
\begin{split}
F(\alpha) - F*\lambda_M(\alpha) = \int_{\T^d}  \brac{F(\alpha) - F(\alpha - \beta)} \lambda_M(\beta)  \intd \beta.
\end{split}
\end{equation*}
Since the F\'ejer kernel is non-negative and has integral 1, the Lipschitz continuity of $F$ gives that
$$
\abs{\int_{\max_i |\beta_i| \leq \eta}  \bigbrac{F(\alpha) - F(\alpha - \beta)} \lambda_M(\beta)  \intd \beta} \leq K \eta.
$$
By the quantitative decay estimate 
$$
 \abs{\int_{\max_i |\beta_i| > \eta}  \bigbrac{F(\alpha) - F(\alpha - \beta)} \lambda_M(\beta)  \intd \beta} \\
 \leq   \frac{2}{M\eta^2}.  
$$
Taking $\eta^3 = 1/(KM)$ then gives
$$
\norm{F - F*\lambda_M}_{\infty} \leq 3K^{2/3} M^{-1/3} .
$$
Setting $M = \ceil{27K^{2}\eps^{-3}}$ we have
$
\norm{F - F*\lambda_M}_{\infty} \leq \eps.
$

We note that we may assume that $\eps \leq 1$, so that $M\ll K^2\eps^{-3}$, for otherwise the result is immediate on taking $F_\eps = 0$ and $M = 0$.
\end{proof}

\section{A dense model lemma}\label{app:dense-model}

\begin{lemma}[Separating hyperplane theorem]
Let $K \subset \R^n$ be closed and convex and $v \notin K$. Then there exists $\phi \in \R^n$ such that for all $u \in K$ we have $v \cdot \phi > u \cdot \phi$.
\end{lemma}
\begin{proof}
See\\
\url{https://en.wikipedia.org/wiki/Hyperplane_separation_theorem}.\end{proof}
\begin{lemma}\label{lem:dual}
For $f,\phi: [N] \to \R$ write
$$
\norm{f} := \bignorm{\hat{f}}_\infty \quad \text{and} \quad \norm{\phi}^* := \sup_{\norm{f} \leq 1} \biggabs{\sum_x f(x)\phi(x)}.
$$ 
Then for any $f,\phi, \psi: [N] \to \R$ we have
\begin{itemize}
\item $\abs{\sum_x f(x)\phi(x)} \leq \norm{f} \norm{\phi}^*$;
\item$\norm{\phi\psi}^* \leq\norm{\phi}^*\norm{\psi}^*$;
\item $  \norm{\phi}_\infty \leq \norm{\phi}^*$.
\end{itemize}
\end{lemma}
\begin{proof}
The first inequality follows from the definition of $\norm{\cdot}^*$. 

Let $e_\alpha$ denote the map $x \mapsto e(\alpha x)$. Then $\norm{\cdot}$ is invariant under multiplication by $e_\alpha$, so for any $f,\phi: [N] \to \R$ the first inequality gives that
$$
\abs{\widehat{f\phi}(\alpha)} = \abs{\sum_x f(x)e_\alpha(x) \phi(x)} \leq \norm{f e_\alpha} \norm{\phi}^* = \norm{f} \norm{\phi}^*.
$$
Hence
\[
\abs{\sum_x f(x)\phi(x)\psi(x)} \leq  \norm{f\phi}\norm{\psi}^* \leq \norm{f} \norm{\phi}^*\norm{\psi}^*.
\]
The second inequality follows.

Suppose that $\norm{\phi}_\infty = 1$, so that $|\phi(x)| = 1$ for some $x\in [N]$. Notice that the function $f := 1_{\set{x}}$ has Fourier transform bounded in magnitude by $1$. Therefore 
$$
\norm{\phi}^* \geq \abs{ \sum_y f(y) \phi(y)} = 1 = \norm{\phi}_\infty.
$$
The third inequality then follows on renormalising.
\end{proof}
\begin{proof}[Proof of Lemma \ref{lem:dense-model}]
We closely follow Conlon and Gowers \cite[Lemma 2.6]{ConlonGowersCombinatorial}. Notice that  
$$
(1+\tfrac{\eps}{2})^{-1}(f_1, \dots, f_r)
$$
 is not a member of the closed convex set
$$
\set{(g_1+h_1, \dots, g_r+h_r) : g_i \geq 0,\ g_1 + \dots + g_r \leq 1_{[N]},\ \bignorm{\hat{h}_i}_\infty \leq \trecip{4}\eps N}.
$$
Hence by the separating hyperplane theorem, there exists $(\phi_1, \dots , \phi_r)$ such that for any $g_i \geq 0$ with $g_1 + \dots + g_r \leq 1_{[N]}$ and $\bignorm{\hat{h}_i}_\infty \leq \trecip{4}\eps N$ we have
\begin{equation}\label{eq:MinMax}
\brac{1+\trecip{2} \eps}^{-1}\sum_i \sum_x f_i(x)\phi_i(x) >  \sum_i\sum_x g_i(x)\phi_i(x) +\sum_i\sum_x h_i(x)\phi_i(x).
\end{equation}
Taking all $g_i$ and $h_i$ zero shows that the left-hand side of \eqref{eq:MinMax} is positive, so we may renormalise $(\phi_1, \dots, \phi_r)$ to give
\begin{equation}\label{eq:MinMax3/2}
\sum_i \sum_x f_i(x)\phi_i(x) = (1+\trecip{2} \eps)N
\end{equation}
and for all $g_i \geq 0$ with $g_1 + \dots + g_r \leq 1_{[N]}$ and $\bignorm{\hat{h}_i}_\infty \leq \trecip{4}\eps N$ we have
\begin{equation}\label{eq:MinMax2}
\sum_i\sum_x g_i(x)\phi_i(x) +\sum_i\sum_x h_i(x)\phi_i(x) < N.
\end{equation}

Notice that
\begin{multline}\label{eq:MinMax4}
 \sum_i \sum_x f_i(x)\phi_i(x) \leq \sum_i \sum_x f_i(x)\max\set{\phi_1(x), \dots, \phi_r(x), 0}\\ \leq \sum_x \nu(x) \max\set{\phi_1(x), \dots, \phi_r(x), 0}.
\end{multline}
For each $x \in [N]$ fix $i(x) \in [r]$ such that $$\max\set{\phi_1(x), \dots, \phi_r(x)}=\phi_{i(x)}(x).$$
Define
$$
g_i(x) := \begin{cases} 1 & \text{if } i = i(x) \text{ and } \phi_i(x) \geq 0,\\
0 & \text{otherwise.}
\end{cases}
$$
By substituting this function into \eqref{eq:MinMax2} with $h_i = 0$, and using both  \eqref{eq:MinMax3/2} and \eqref{eq:MinMax4}, we deduce that
\begin{equation*}
 \sum_x \sqbrac{\nu(x)-1_{[N]}(x)} \max\set{\phi_1(x), \dots, \phi_r(x), 0}  > \tfrac{1}{2} \eps N .
\end{equation*}

Using the notation and content of Lemma \ref{lem:dual}, the inequality  gives that  
\begin{equation}\label{eq:phi-dual-bd}
\norm{\phi_i}_\infty \leq \norm{\phi_i}^* \leq 4/\eps.
\end{equation}
By the Stone--Weierstrass theorem\footnote{\url{https://en.wikipedia.org/wiki/Stone-Weierstrass_theorem}}, there exists a polynomial $P_\eps$ with degree and coefficients of size $O_{\eps,r}(1)$ such that for all $|x_i| \leq 4/\eps$ we have 
$$
\abs{\max\set{x_1, \dots, x_r, 0} - P_\eps(x_1, \dots, x_s)}\leq \eps/100.
$$
Notice that we may assume that $\sum_x \nu(x) \leq 2N$, otherwise we are done. Hence
\begin{equation*}
 \sum_x \sqbrac{\nu(x)-1_{[N]}(x)} P_\eps\brac{\phi_1(x), \dots, \phi_r(x)}  > \tfrac{1}{4} \eps N .
\end{equation*}
 Expanding the polynomial, and applying the pigeon-hole principle, there exist $\psi_1, \dots, \psi_R \in \set{\phi_1,\dots \phi_r}$ with $R \ll_{\eps,r} 1$ such that 
\begin{equation*}
 \abs{\sum_x \sqbrac{\nu(x)-1_{[N]}(x)}\psi_1(x)\dotsm\psi_R(x)}  \gg_{\eps,r} N.
\end{equation*}

Recalling \eqref{eq:phi-dual-bd} and Lemma \ref{lem:dual} we have $
\norm{\psi_1 \dotsm \psi_R}^* \ll_{\eps,r} 1.
$
Hence, again applying the first inequality in Lemma \ref{lem:dual}, we deduce \eqref{eq:non-unif}.
\end{proof}


\begin{thebibliography}{CFSZ20}

\bibitem[BHP01]{BHPDifference}
R.~C. Baker, G.~Harman, and J.~Pintz.
\newblock The difference between consecutive primes. {II}.
\newblock {\em Proc. London Math. Soc. (3)}, 83(3):532--562, 2001.

\bibitem[Bou99]{BourgainTriples}
J.~Bourgain.
\newblock On triples in arithmetic progression.
\newblock {\em Geom. Funct. Anal.}, 9(5):968--984, 1999.

\bibitem[CFSZ20]{CFSZRegularity}
D.~Conlon, J.~Fox, B.~Sudakov, and Y.~Zhao.
\newblock The regularity method for graphs with few 4-cycles.
\newblock {\em ArXiv e-prints}, 2020.

\bibitem[CFSZ21]{CFSZWhich}
\bysame,
\newblock Which graphs can be counted in $C_4$-free graphs? 
\newblock {\em ArXiv e-prints}, 2021.

\bibitem[CG16]{ConlonGowersCombinatorial}
D.~Conlon and W.~T. Gowers.
\newblock Combinatorial theorems in sparse random sets.
\newblock {\em Ann. of Math. (2)}, 184(2):367--454, 2016.

\bibitem[Cil00]{CillerueloGaps}
J.~Cilleruelo.
\newblock Gaps in dense {S}idon sets.
\newblock {\em Integers}, 0:A11, 6pp, 2000.

\bibitem[Ebe21]{eberhard2021apparent}
S.~Eberhard.
\newblock The apparent structure of dense {S}idon sets.
\newblock {\em Talk given at CANT}, 2021. \url{https://youtu.be/s4ItIkkUvF4}

\bibitem[EM21]{eberhardmanners2021apparent}
S.~Eberhard and F.~Manners.
\newblock The apparent structure of dense {S}idon sets.
\newblock {\em ArXiv e-prints}, 2021.

\bibitem[EF91]{ErdosFreudSums}
P.~Erd\H{o}s and R.~Freud.
\newblock On sums of a {S}idon-sequence.
\newblock {\em J. Number Theory}, 38(2):196--205, 1991.

\bibitem[ET41]{ErdosTuranProblem}
P.~Erd\H{o}s and P.~Tur\'{a}n.
\newblock On a problem of {S}idon in additive number theory, and on some
  related problems.
\newblock {\em J. London Math. Soc.}, 16:212--215, 1941.

\bibitem[FGR88]{FGRQuantitative}
P.~Frankl, R.~L. Graham, and V.~R\"{o}dl.
\newblock Quantitative theorems for regular systems of equations.
\newblock {\em J. Combin. Theory Ser. A}, 47(2):246--261, 1988.

\bibitem[FK21]{forey2021algebraic}
A.~Forey and E.~Kowalski.
\newblock Algebraic curves in their {J}acobian are {S}idon sets.
\newblock {\em ArXiv e-prints}, 2021.

\bibitem[Gow12]{GowersDenseSidon}
W.~T. Gowers.
\newblock What are dense {S}idon subsets of $\{1,2,{\ldots},n\}$ like?
\newblock Blogpost available at \url{https://bit.ly/3xHq2NM}, 2012.

\bibitem[GRS90]{GRSRamsey}
R.~L. Graham, B.~L. Rothschild, and J.~H. Spencer.
\newblock {\em Ramsey theory}.
\newblock Wiley-Interscience Series in Discrete Mathematics and Optimization.
  John Wiley \& Sons, Inc., New York, second edition, 1990.
\newblock A Wiley-Interscience Publication.

\bibitem[KSV13]{ksv13}
P.~Keevash, B.~Sudakov, and J.~Verstra\"{e}te. 
\newblock On a conjecture of Erdős and Simonovits: Even cycles. 
\newblock {\em Combinatorica}, 33, 699--732, 2013.

\bibitem[Kol99]{KolountzakisUniform}
M.~N. Kolountzakis.
\newblock On the uniform distribution in residue classes of dense sets of
  integers with distinct sums.
\newblock {\em J. Number Theory}, 76(1):147--153, 1999.

\bibitem[Lin98]{LindstromWell}
B.~Lindstr\"{o}m.
\newblock Well distribution of {S}idon sets in residue classes.
\newblock {\em J. Number Theory}, 69(2):197--200, 1998.

\bibitem[Mon94]{MontgomeryTen}
H.~L. Montgomery.
\newblock {\em Ten lectures on the interface between analytic number theory and
  harmonic analysis}, volume~84 of {\em CBMS Regional Conference Series in
  Mathematics}.
\newblock American Mathematical Society, Providence, RI, 1994.

\bibitem[O'B04]{OBryantComplete}
K.~O'Bryant.
\newblock A complete annotated bibliography of work related to {S}idon
  sequences.
\newblock {\em Electron. J. Combin.}, (DS11), 2004.

\bibitem[Rot54]{RothCertainII}
K.~F. Roth.
\newblock On certain sets of integers. {II}.
\newblock {\em J. London Math. Soc.}, 29:20--26, 1954.

\bibitem[Sin38]{SingerTheorem}
J.~Singer.
\newblock A theorem in finite projective geometry and some applications to
  number theory.
\newblock {\em Trans. Amer. Math. Soc.}, 43(3):377--385, 1938.

\bibitem[TV06]{TaoVuAdditive}
T.~Tao and V.~Vu.
\newblock {\em Additive Combinatorics}, volume 105 of {\em Cambridge Studies in
  Advanced Mathematics}.
\newblock Cambridge University Press, Cambridge, 2006.

\end{thebibliography}

\end{document}